\documentclass{amsart}
\usepackage{graphicx,amsfonts,amssymb,amsmath,amsthm,amscd}
\usepackage[all]{xy}
\newtheorem{thm}{Theorem}[section]
\newtheorem{cor}[thm]{Corollary}
\newtheorem{lem}[thm]{Lemma}
\newtheorem{prop}[thm]{Proposition}
\theoremstyle{definition}
\newtheorem{defn}[thm]{Definition}
\newtheorem{eg}[thm]{Example}
\theoremstyle{remark}
\newtheorem{rem}[thm]{Remark}
\numberwithin{equation}{section}

\newcommand{\p}{\frak{p}}
\newcommand{\q}{\frak{q}}

\newcommand{\Hom}{\mathrm{\,Hom}}

\newcommand{\dlim}{{\displaystyle\lim_{\stackrel{\longrightarrow}{\scriptscriptstyle{n\in\Bbb{N}}}}}}
\newcommand{\Spec}{\mathrm{Spec}}
\newcommand{\Ann}{\mathrm{Ann}}

\newcommand{\Supp}{\mathrm{Supp}}
\newcommand{\rad}{\mathrm{rad}}
\newcommand{\Ext}{\mathrm{\, Ext}}
\begin{document}

\title[Prime Submodules And A Sheaf On The Prime Spectra Of Modules]{Prime Submodules And A Sheaf On The Prime Spectra Of Modules}%
\author[D. Hassanzadeh]{D. Hassanzadeh-lelekaami$^*$}%
\address{Department of Pure Mathematics, Faculty of Mathematical Sciences, University of Guilan, P. O. Box 41335-19141, Rasht, Iran.}%
\email{dhmath@guilan.ac.ir and lelekaami@gmail.com}%
\author[H. Roshan]{H. Roshan-Shekalgourabi}%
\address{Department of Pure Mathematics, Faculty of Mathematical Sciences, University of Guilan, P. O. Box 41335-19141, Rasht, Iran.}%
\email{hroshan@guilan.ac.ir  and hrsmath@gmail.com}%

\thanks{$^*$ Corresponding author}%
\date{\today}%
\dedicatory{Dedicated to Professor Rahim Zaare Nahandi (University of Tehran)}%
\begin{abstract}
We define and investigate  a sheaf of modules on the prime spectra of   modules and it is shown that there is an isomorphism between the sections of this sheaf and the ideal transform module.
\end{abstract}
\maketitle

\section{INTRODUCTION}

Over the past several decades, a theory of prime submodules has developed which generalizes the notion of prime ideals in commutative rings. A fairly large number of articles have considered the extent to which properties of the prime ideals of a ring have counterparts for the prime submodules of a module. For example, in 1984, some basic properties of prime submodules were compiled \cite{lu84}. In 1986, McCasland and Moore adjusted the definition to the $M$-radical given in \cite[p. 37]{mm96}. Thus, began the extensive study of radical theory for modules, which has continued with the more recent work of many algebraists \cite{abbhass2012, lu89, lu07, man99, mm91, ps02}. In work beginning in the nineties, many authors have studied analogs of Krull's principal ideal theorem or/and developed a dimension theory for modules \cite{as99, gms94}. Other analogs for well-known theorems, such as Cohen's theorem and the prime avoidance theorem, were given in some papers \cite{lu84, lu97}. Also, some new classes of modules  such as weak multiplication modules, primeful modules, top modules and strongly top modules are defined by notion of prime submodule (see \cite{abbhass2012, az03, lu99, mms97}). The concept of prime submodule has led to  the development of topologies on the spectrum of modules. Topologies are considered by Duraivel, McCasland, Moore, Smith, and Lu  in \cite{dur94, lu99, mms97}. In the literature, there are many papers devoted to the Zariski topology on the spectrum of modules \cite{abbhass2012, ah10, ja11, behI08,  lu10, mms98}. It is well-known that Zariski topology on the spectrum of prime ideals of a ring is one of the main tools in algebraic geometry. The spectrum of a ring, equipped with a  sheaf of rings, is what Grothendieck called an affine scheme. However, many aspects of Zariski topology on the spectrum of prime submodules of a module is investigated, but there are some natural questions about this topology. For example, is it possible to define a sheaf of modules (or rings) on the prime spectra of modules similar to the the sheaf of rings $\mathcal{O}_{\Spec(R)}$ on the topological space $\Spec(R)$? Some explorations  to find a response to this question were done in \cite{tek05} and \cite{ah13}. Sheaf theory provides a language for the discussion of geometric objects of many different kinds. At present it finds its main applications in topology and in modern algebraic geometry, where it has been used with great success as a tool in the solution of several longstanding problems. This importance is the motivation  for us to define a  sheaf of modules $\mathcal{A}(N,M)$ on the topological space $\Spec(M)$, the set of all prime submodules of  $M$ equipped with Zariski topology (where $M$ and $N$ are two $R$-modules).

Section 2 presents some preliminaries. In Section 3  our main results are stated and proved. The results are illustrated by examples. In the Proposition~\ref{prop.stalk}, we are interested in finding the stalk of  the sheaf $\mathcal{A}(N,M)$. In the Lemma~\ref{firstlemma}, we examine the sheaf $\mathcal{A}(N,M)$ in case the ring $R$ is Noetherian. In order to find further results about the sheaf $\mathcal{A}(N, M)$, one of the major aims for this paper is a description of the sheaf $\mathcal{A}(N, M)$ in terms of the ideal transform module (see Theorem~\ref{mainloc}).

The local cohomology theory introduced by Grothendieck is a useful tool for attacking problems in commutative algebra and algebraic geometry. In the Corollary~\ref{corloc}, we show that there is an isomorphism between the right derived functors of the sheaf mentioned above  and the local cohomology modules. The theory of schemes is the foundation for algebraic geometry formulated by  Grothendieck and his many coworkers. We use this as motivation for the Theorem~\ref{scheme}, in which we will  introduce a scheme over the prime spectra of modules.

\section{PRELIMINARIES}

Throughout this paper, all rings are commutative with identity and all modules are unital. For a submodule $N$ of an $R$-module $M$, $(N :_R M)$ denotes the ideal $\{r\in R
\mid rM\subseteq N\}$ and \emph{annihilator} of $M$, denoted by $\Ann_R(M)$,
is the ideal $(\textbf{0}:_R M)$.  If there is no ambiguity we will consider $(N : M)$ (resp. $\Ann(M)$) instead of $(N :_R M)$ (resp. $\Ann_R(M)$). The $R$-module $M$ is called \emph{faithful} if $\Ann(M)=(0)$. A submodule $N$ of an $R$-module $M$ is said to be \emph{prime} if $N\neq M$ and
whenever $rm \in N$ (where $r\in R$ and $m \in M$) then $r\in (N
: M)$ or $m \in N$. If $N$ is prime, then   $\p = (N : M)$
is a prime ideal of $R$. In this case,  $N$ is said to be
\emph{$\p$-prime} (see \cite{lu84}). The concept of prime submodules of a module is studied by many algebraists, for example see \cite{abbhass2012, abu95, ja11, az03, Dau78,   lu95, lu99, lu03, lu07, lu10, mm92, mms97, mms98, mms02}. The set of all prime submodules of an $R$-module $M$ is called the \emph{prime spectrum} of $M$ and denoted by
$\Spec(M)$. Similarly, the collection of all $\p$-prime submodules
of an $R$-module $M$ for any $\p\in \Spec(R)$ is designated by $\Spec_{\p}(M)$. We
remark that $\Spec(\textbf{0}) = \emptyset$ and that $\Spec(M)$ may
be empty for some nonzero $R$-module $M$. For example,
$\mathbb{Z}_{p^{\infty}}$ as a $\mathbb{Z}$-module has no prime
submodule for any prime integer $p$ (see \cite{lu95}). Such a
module is said to be \emph{primeless}. \textbf{In the remainder of the paper, $M$ is a non-primeless $R$-module and  $X$ denotes the prime spectrum $\Spec(M)$ of $M$}. An $R$-module $M$ is called
\emph{primeful} if either $M =(\textbf{0})$ or $M\neq (\textbf{0})$ and
the map $\psi : \Spec(M)\rightarrow \Spec(R/\Ann(M))$ defined by
$\psi(P) = (P : M)/\Ann(M)$ for every $P\in \Spec(M)$, is
surjective (see \cite{lu07}).  Suppose that $M$ is an $R$-module.  For any submodule $N$ of $M$, $V(N)$ is defined as $\{P \in X \mid (P:M)\supseteq (N:M) \}$ (see \cite{lu99}). Set $\textsf{Z}(M)=\{ V(N)\, |\, N \leq M \}$. Then the elements of the set $\textsf{Z}(M)$ satisfy the axioms for closed sets in a topological space $X$ (see \cite{lu99}). The resulting topology due to $\textsf{Z}(M)$ is called the \emph{Zariski topology relative to $M$} and denoted by $\tau $.  Zariski topology on the spectrum of prime ideals of a ring is one of the main tools in algebraic geometry. In the literature, there are many different generalizations of the Zariski topology of rings to modules via prime submodules (see \cite{ah10, behI08, lu99, mms97}). We recall that for any element $r$ of a ring $R$, the set $D_r = \Spec(R)\backslash V(Rr)$ is open in $\Spec(R)$ and the family  $\{D_r |\, r\in R\}$ forms a base for the Zariski topology on $\Spec(R)$. Let M be an R-module. For each $r\in R$, we define $X_r = \Spec(M)\backslash V (rM)$. It is shown in \cite[Proposition 4.3]{lu99} that the set $\{X_r |\, r \in R\}$ forms a base for the Zariski topology on $(\Spec(M),\tau)$.

\begin{rem}
Let $N$ be an $R$-module. For an ideal $I$ of $R$ we recall that the \emph{$I$-torsion submodule of $N$} is $\Gamma_I(N)=\bigcup_{n\geq 1}(0:_N I^n)$ and $N$ is said to be \emph{$I$-torsion} if $N=\Gamma_I(N)$. The \emph{$i$-th local cohomology module of $N$ with respect to $I$} is defined as $$H^i_I(N) = \dlim \Ext^i_R(R/I^n, N).$$ The reader can refer to \cite{sb98} for the basic properties of local cohomology modules. We also recall that the \emph{ideal transform of $N$ with respect to $I$}, (or, alternatively, the $I$-transform of $N$) is defined as $$D_I(N)=\dlim \Hom_R(I^n, N).$$
\end{rem}

\begin{rem} \label{t00}(\cite[Theorem 6.1]{lu99})
The following statements are equivalent: (1) $(X,\tau)$ is a $T_0$-space; (2) The natural map $\psi: \Spec(M)\rightarrow \Spec(R/ \Ann(M))$ is injective; (3) If $V (P) = V (Q)$, then $P = Q$ for any $P,Q \in  \Spec(M)$; (4) $|\Spec_{\p}(M)|\leq 1$ for every $\p \in \Spec(R)$.\label{t00p}
\end{rem}

\begin{rem}
We recall that the \emph{Zariski radical} of a submodule $N$ of an $R$-module $M$, denoted by $z\rad(N)$, is the intersection of all members of $V(N)$, that is, $z\rad(N)=\cap_{P\in V(N)} P$ (see \cite[Definitions 1.3]{lu10}).
\end{rem}

\section{MAIN RESULTS}

First, we use  the concept of prime spectrum of an $R$-module $M$ to define a sheaf of modules on  $\Spec(M)$.  For every open subset $U$ of $X$ we set $\Supp(U)=\{(P:M)\mid P\in U\}$.

\begin{defn}
Let  $N$ be an $R$-module. We define the \emph{sheaf associated to $N$  relative to $M$}, denoted by $\mathcal{A}(N,M)$, as follows: for every open subset $U$ of $X$  we define $\mathcal{A}(N,M)(U)$ to be the set of all elements $(\gamma_\p)_{\p\in \Supp(U)}\in \coprod_{\p\in \Supp(U)} N_\p $ in which for each  $Q\in U$,  there exists an open neighborhood  $W\subseteq U$  of $Q$  and  $s\in R$, $m\in N$ such that, for each $P \in W$,  it is the case that   $s\not\in \p:=(P:M)$ and $\gamma_\p=m/s\in N_\p$.
\end{defn}

We denote the sheaf of rings $\mathcal{A}(R,M)$ by $\mathcal{A}_X$. We can define a map $\epsilon_N^U:N\longrightarrow \mathcal{A}(N,M)(U)$
 by $\epsilon_N^U(n)=(n/1)_{\p\in \Supp(U)}$ for all $n\in N$. We note that $\epsilon_N^U$ is an $R$-homomorphism. Now, let $U$ and $V$ be open subsets of $X$ with $\emptyset \neq V \subseteq U$, and let $\gamma=(\gamma_\p)_{\p\in \Supp(U)}\in \mathcal{A}(N,M)(U)$. Then it is clear that the restriction $(\gamma_\p)_{\p\in \Supp(V)}$ belongs to $\mathcal{A}(N,M)(V)$. We therefore have the restriction map
$\rho_{UV}:\mathcal{A}(N,M)(U)\longrightarrow \mathcal{A}(N,M)(V)$ for which
$\rho_{UV}(\gamma)=(\gamma_\p)_{\p\in \Supp(V)}$  for all $\gamma\in \mathcal{A}(N,M)(U)$. Of course, we define $\rho_{U\emptyset}:\mathcal{A}(N,M)(U)\longrightarrow \mathcal{A}(N,M)(\emptyset)=0$
to be the zero map. It is easy to see that  $\rho_{UV} \circ \epsilon_N^U=\epsilon_N^V$. We can make $\mathcal{A}(N,M)$ into a sheaf by using the restriction maps (as mentioned above). The reader can refer to \cite{ hart, Liu02, Tennison} for the basic properties of sheaves. Since another way to describe a sheaf is by its stalks, in the next proposition we find the stalk of the sheaf $\mathcal{A}(N,M)$.

\begin{prop}\label{prop.stalk}
Let  $N$ be an $R$-module.  For each $P\in X$, the stalk $\mathcal{A}(N,M)_{P}=\varinjlim_{P\in U}\mathcal{A}(N,M)(U)$ of the sheaf $\mathcal{A}(N,M)$ is isomorphic to $N_\p$, where $\p:=(P:M)$.
\end{prop}
\begin{proof}
Let $P$ be a $\p$-prime submodule of $M$ and $m\in \mathcal{A}(N,M)_P$. Then there exists a neighborhood $U$  of $P$ and $s=(s_\p)_{\p\in\Supp(U)}\in \mathcal{A}(N,M)(U)$ such that $s$ represents $m$. We define $\varphi:\mathcal{A}(N,M)_P\rightarrow N_{\p}$ by $\varphi(m)=s_\p$. Let $V$ be another neighborhood of $P$ and $t=(t_\p)_{\p\in\Supp(V)}\in \mathcal{A}(N,M)(V)$ such that $t$ also represents $m$. Then there exists an open subset $W\subseteq U\cap V$ such that $P\in W$ and $s|_W=t|_W$. Since $P\in W$, $s_\p=t_\p$. We claim that $\varphi$ is an isomorphism.

Let $x\in N_\p$. Then $x=a/f$ where $a\in N$ and $f\in R\setminus \p$. Since $f\not\in \p$, $P\in X_f$. Now we define $s_\q=a/f$ in $N_{\q}$ for all $Q\in X_f$, where $\q:=(Q:M)$. Then $s=(s_\q)_{\q\in\Supp(X_f)}\in \mathcal{A}(N,M)(X_f)$. If $m$ is the equivalent class of $s$ in $\mathcal{A}(N,M)_P$, then $\varphi(m)=x$. Hence $\varphi$ is surjective.

Now, let $m\in \mathcal{A}(N,M)_P$ and $\varphi(m)=0$. Let $U$ be an open neighborhood of $P$ and $s=(s_\p)_{\p\in\Supp(U)}\in \mathcal{A}(N,M)(U)$ such that $s$ represents $m$. There is an open neighborhood $V\subseteq U$ of $P$ and there are elements $a\in N$ and $f\in R$ such that for all $Q\in V$, it is the case that $f\not\in \q:=(Q:M)$ and $s_\q=a/f\in N_{\q}$. Thus $V\subseteq X_f$. Then $0=\varphi(m)=s_\p=a/f$ in $N_\p$. So, there is $h\in R\setminus \p$ such that $ha=0$. For $Q\in X_{fh}=X_f\cap X_h$ we have $s_\q=a/f\in N_\q$. Since $h\not\in \q$, $s_\q=\frac{a}{f}=\frac{h}{h}\frac{a}{f}=0$. Thus $s|_{X_{fh}}=0$. Therefore, $s=0$ in $\mathcal{A}(N,M)(X_{fh})$. Consequently $m=0$. This completes the proof.
\end{proof}

\begin{eg}
Consider the $\mathbb{Z}$-modules $M=(\mathbb{Z}/2\mathbb{Z})\times (\mathbb{Z}/3\mathbb{Z})$ and $N=\mathbb{Z}/14\mathbb{Z}$. It is clear that $\Spec(M)=\{2M, 3M\}$. By Proposition~\ref{prop.stalk}, we have $\mathcal{A}(\mathbb{Z}/14\mathbb{Z}, (\mathbb{Z}/2\mathbb{Z})\times (\mathbb{Z}/3\mathbb{Z}))_{2M}=(\mathbb{Z}/14\mathbb{Z})_{2\mathbb{Z}}$ and $\mathcal{A}(\mathbb{Z}/14\mathbb{Z}, (\mathbb{Z}/2\mathbb{Z})\times (\mathbb{Z}/3\mathbb{Z}))_{3M}=(0)$.
\end{eg}

\begin{cor}\label{cor.stalk1}
For each $P\in X$, the stalk $\mathcal{A}_{X, P}$ of the sheaf of rings $\mathcal{A}_X$ is isomorphic to $R_\p$, where $\p:=(P:M)$. Moreover, $(X, \mathcal{A}_{X})$  is a locally ringed space.
\end{cor}
\begin{proof}
Use Proposition~\ref{prop.stalk}.
\end{proof}

In order to find further results about the sheaf $\mathcal{A}(N, M)$, the next major aim for this paper is a description of the sheaf $\mathcal{A}(N, M)$ in terms of the ideal transform module, and we now embark on the preparations for this result. Here, we present a lemma that shows the sheaf $\mathcal{A}(N, M)$ has a simple expression when $R$ is a Noetherian ring.

\begin{lem}\label{firstlemma}\label{20.1.8i}
Let $R$ be a Noetherian ring and  $N$ be an $R$-module. Let $M$ be a primeful $R$-module and $U:=X\setminus V(K)$, where $K\leq M$. Then for each $\gamma=(\gamma_\p)_{\p\in \Supp(U)}\in \mathcal{A}(N,M)(U)$, there exist $r\in \mathbb{N}$, $s_1, \ldots, s_r\in (K:M)$ and $m_1, \ldots, m_r\in N$ such that  $U=\bigcup_{i=1}^r X_{s_i}$ and $\gamma_\p=m_i/s_i$ for all $P\in X_{s_i}$ (for $i=1, \ldots, r$), where $\p:=(P:M)$.
\end{lem}
\begin{proof}
Since $R$ is Noetherian, $U$ is quasi-compact by \cite[Proposition 3.24]{abbhass2012}. Thus there exist  $n\in \mathbb{N}$, open subsets $W_1, \ldots, W_n$ of $U$, $t_1, \ldots, t_n\in R$, $l_1, \ldots, l_n\in N$ such that $$U=\bigcup_{j=1}^n W_j$$ and for each $j=1, \ldots, n$ and $P\in W_j$ we have $t_j\not\in \p:=(P:M)$ and $\gamma_\p=l_j/t_j$. Fix $j\in \{1, \ldots, n\}$. Since $W_j$ is an open subset of $U$, there is a submodule $H_j$ of $M$ such that $W_j=X\setminus V(H_j)$. Now, $W_j\subseteq U=X\setminus V(K)$, and so $V(K)\subseteq V(H_j)$. Since $R$ is Noetherian, $$(H_j:M)=Rh_{j1}+ \cdots+ Rh_{jn_j}$$ for some $h_{j1}, \ldots, h_{jn_j}\in R$. This implies that
      \begin{eqnarray*}
        W_j &=& X\setminus V(H_j) \\
         &=& X\setminus V((H_j:M)M) \text{ by \cite[Result 3]{lu99}}\\
         &=& X\setminus V((Rh_{j1}+ \cdots+ Rh_{jn_j})M) \\
         &=&  X\setminus \bigcap_{f=1}^{n_j} V(h_{jf}M) \\
         &=& \bigcup_{f=1}^{n_j} X_{h_{jf}}.
      \end{eqnarray*}
      On the other hand, we have
      \begin{eqnarray*}
        V(K)\subseteq V(H_j) &\Rightarrow& z\rad(H_j)\subseteq z\rad(K) \\
         &\Rightarrow& (z\rad(H_j):M)\subseteq (z\rad(K):M) \\
         &\Rightarrow& \sqrt{(H_j:M)}\subseteq \sqrt{(K:M)} \text{ by \cite[Propositin 2.3]{lu10}}.
      \end{eqnarray*}
Hence, there exists $d\in \mathbb{N}$ such that $$(H_j:M)^d\subseteq \left(\sqrt{(H_j:M)}\right)^d\subseteq \left(\sqrt{(K:M)}\right)^d\subseteq (K:M).$$ It follows that $h_{jf}^d\in (K:M)$, for each $f\in \{1, \ldots, n_j\}$. This means that for each $f\in \{1, \ldots, n_j\}$ and for each $P\in X_{h_{jf}}$, we have $t_jh_{jf}^d\not\in (P:M)$. We deduce that  $$X_{h_{jf}}=X\setminus V(h_{jf}M)=X\setminus V(t_jh_{jf}^dM)=X_{t_jh_{jf}^d},$$ and we can write $$\gamma_{_{(P:M)}}=\frac{l_j}{t_j}=\frac{l_jh_{jf}^d}{t_jh_{jf}^d}\in N_{(P:M)}.$$  We have therefore only to relabel the pairs $(t_jh_{jf}^d, l_jh_{jf}^d)\in (K:M)\times M$ ($f=1, \ldots, n_j, j=1, \ldots, n$) in order to complete the proof.
\end{proof}

\begin{lem}\label{gammazero}\label{20.1.8ii}
Let $N$ be an $R$-module. Set $U:=X\setminus V(K)$, where $K\leq M$.  Then $\Gamma_{(K:M)}(\mathcal{A}(N,M)(U))=0$.
\end{lem}
\begin{proof}
Let $\gamma=(\gamma_\p)_{\p\in \Supp(U)}\in \Gamma_{(K:M)}(\mathcal{A}(N,M)(U))$. There exists $n\in \mathbb{N}$ such that $(K:M)^n\gamma=0$. Consider $\p\in \Supp(U)$. Thus, $(K:M)^n\not\subseteq \p$. So, there exists $s_\p \in (K:M)^n\setminus \p$. Since $s_\p\gamma=0$, we have $s_\p\gamma_\p=0$ and so $$\gamma_\p=\frac{1}{s_\p}(s_\p\gamma_\p)=0\in N_\p.$$ Hence, $\gamma=0$.
\end{proof}

\begin{eg}
Consider the $\mathbb{Z}$-module $M=\mathbb{Z}\oplus (\mathbb{Z}/64\mathbb{Z})$. Let $U=\Spec(M)\setminus V(\mathbb{Z}\oplus (0))$ and $N$ be an arbitrary $\mathbb{Z}$-module. Then by Lemma~\ref{gammazero}, $\Gamma_{I}(\mathcal{A}(N, \mathbb{Z}\oplus (\mathbb{Z}/64\mathbb{Z}))(U))=0$, for any ideal $I$ of $\mathbb{Z}$ with $I\supseteq 64\mathbb{Z}$.
\end{eg}

\begin{prop}\label{kernel}\label{20.1.8iii}
Let $R$ be a Noetherian ring, $M$ be a faithful  primeful $R$-module and $N$ be an $R$-module. Set $U=X\setminus V(K)$, where $K\leq M$.  Then $\Gamma_{(K:M)}(N)=\ker (\epsilon_N^U)$, and so $\ker (\epsilon_N^U)$ is $(K:M)$-torsion.
\end{prop}
\begin{proof}
Since $\epsilon_N^U:N\longrightarrow \mathcal{A}(N,M)(U)$ is an $R$-homomorphism it follows  from Lemma~\ref{gammazero} that $$\epsilon_N^U(\Gamma_{(K:M)}(N))\subseteq \Gamma_{(K:M)}(\mathcal{A}(N,M)(U))=0,$$  and so $\Gamma_{(K:M)}(N)\subseteq \ker (\epsilon_N^U)$.

Now, suppose that $m\in \ker (\epsilon_N^U)$. Then $\frac{m}{1}=0 \in N_\p$ for all $\p\in \Supp(U)$. So, for each $\p\in \Supp(U)$ there is $t_\p\in R\setminus \p$ such that $t_\p m=0$. Put $$J:=\sum_{\p\in \Supp(U)}Rt_\p.$$ Hence, $Jm=0$.

Let $\q\in V(J)$. We claim that $(K:M)\subseteq \q$. On the contrary, suppose that  $(K:M)\nsubseteq \q$. Since $M$ is faithful primeful, there is a prime submodule $Q$ such that $(Q:M)=\q$. Therefore, $Q\in U$ and $\q\in \Supp(U)$. This yields that $t_\q\in J\subseteq \q$. This contradicts the fact that $t_\q\in R\setminus \q$.  Thus, $$(K:M)\subseteq \sqrt{(K:M)}=\bigcap_{\p\in V((K:M))}\p\subseteq \bigcap_{\p\in V(J)}\p=\sqrt{J}.$$ Since $R$ is Noetherian, $(K:M)^n\subseteq J$  for some integer $n\in \mathbb{N}$. Hence, $(K:M)^n m\subseteq Jm=0$. This implies that $m\in \Gamma_{(K:M)}(N)$. This completes the proof.
\end{proof}

\begin{eg}
Consider the $\mathbb{Z}$-modules $N=(\mathbb{Z}/2\mathbb{Z})\times (\mathbb{Z}/3\mathbb{Z})\times (\mathbb{Z}/7\mathbb{Z})$ and $M= \mathbb{Z} \oplus \mathbb{Z}_{p^\infty}$, where $p$ is a prime integer. Let $q$ be a prime integer. Then $qM=q\mathbb{Z} \oplus \mathbb{Z}_{p^\infty}\neq M$. Hence, by \cite[Proposition 2]{lu84}, $qM$ is a $(q)$-prime submodule of $M$. Indeed, $qM$ is a maximal submodule of $M$ and it is easy to see that $qM$ is the only  $(q)$-prime submodule of $M$. On the other hand, the torsion submodule $T(M)$ of $M$ is $(0)\oplus \mathbb{Z}_{p^\infty}$. Hence, by \cite[Lemma 4.5]{lu03}, $T(M)$ is a $(0)$-prime submodule of $M$. Therefore, $\Spec(M)=\{qM \,| \, q \text{ is a prime integer}\}\cup \{(0)\oplus \mathbb{Z}_{p^\infty}\}$. This yields that $M$ is primeful. Put $U=\Spec(M)\setminus V(K)$, where $K=30\mathbb{Z} \oplus \mathbb{Z}_{p^\infty}$ is a submodule of $M$. Then by Proposition~\ref{kernel}, $\ker (\epsilon_N^U)=(\mathbb{Z}/2\mathbb{Z})\times (\mathbb{Z}/3\mathbb{Z})$.
\end{eg}

\textbf{In the remainder of the paper we assume that $M$ is a faithful primeful $R$-module.
}
\begin{prop}\label{2019}
Let $R$ be a Noetherian ring and $N$ be an $R$-module. Set $U=X\setminus V(K)$, where $K\leq M$. Let $W$ be an open subset of $X$ such that $U\subseteq W$. Consider the restriction map  $\rho_{WU}:\mathcal{A}(N,M)(W)\longrightarrow \mathcal{A}(N,M)(U)$. Then $\ker (\rho_{WU})=\Gamma_{(K:M)}(\mathcal{A}(N,M)(W))$.
\end{prop}
\begin{proof}
By Lemma \ref{gammazero}, $$\rho_{WU}(\Gamma_{(K:M)}(\mathcal{A}(N,M)(W)))\subseteq \Gamma_{(K:M)}(\mathcal{A}(N,M)(U))=0.$$  So, $\Gamma_{(K:M)}(\mathcal{A}(N,M)(W))\subseteq \ker (\rho_{WU})$.

There exists $L\leq M$ such that  $W=X\setminus V(L)$. Let $\gamma=(\gamma_\p)_{\p\in \Supp(W)}\in \ker (\rho_{WU})$. By Lemma \ref{firstlemma}, there exist $r\in \mathbb{N}$, $s_1, \ldots, s_r\in (L:M)$ and $m_1, \ldots, m_r\in N$ such that  $W=\bigcup_{i=1}^r X_{s_i}$ and for each $i=1, \ldots, r$ and each $P\in X_{s_i}$, we have $\gamma_{_{(P:M)}}=m_i/s_i$. Since $\gamma \in \ker (\rho_{WU})$, we have $\gamma_\p=0$ for all $\p \in \Supp(U)$. Fix $i\in \{1, \ldots, r\}$. Set $U'=U\cap X_{s_i}$. Hence,
\begin{eqnarray*}
  U' &=& (X\setminus V(K))\cap(X\setminus V(s_iM)) \\
   &=& X\setminus (V(K)\cup V(s_iM)) \\
   &=& X\setminus V(K\cap s_iM).
\end{eqnarray*}
Then $\frac{m_i}{s_i}=0\in N_\p$ for all $\p\in \Supp(U')$. This means that $\epsilon_N^{U'}(m_i)=0$, so that there exists $h_i\in \mathbb{N}$ such that $(K\cap s_iM:M)^{h_i}m_i=0$ by Proposition \ref{kernel}. Let $h:=\max \{h_1, h_2, \ldots, h_r\} $. Now, let $\p\in \Supp(W)$. There exists $i\in \{1, \ldots, r\}$ with $\p\in \Supp(X_{s_i})$. Let $d\in (K:M)^h$. Since $s_i^h\in (s_iM:M)^h$, we have $$ds_i^h\in (K:M)^h(s_iM:M)^h\subseteq ((K:M)\cap(s_iM:M))^h=(K\cap s_iM:M)^h.$$ Therefore, we deduce that $$d\gamma_\p=\frac{dm_i}{s_i}=\frac{ds_i^hm_i}{s_i^{h+1}}=0\in N_\p,$$  for all $d\in (K:M)^h$. This implies that $(K:M)^h\gamma=0$ and $\gamma \in \Gamma_{(K:M)}(\mathcal{A}(N,M)(W))$.
\end{proof}

\begin{prop}\label{20.1.10}
Let $R$ be a Noetherian ring and $N$ be an $R$-module. Set $U=X\setminus V(K)$, where $K\leq M$. Then $\epsilon_N^U:N\longrightarrow \mathcal{A}(N,M)(U)$ has a $(K:M)$-torsion cokernel.
\end{prop}
\begin{proof}
Let $\gamma=(\gamma_\p)_{\p\in \Supp(U)}\in \mathcal{A}(N,M)(U)$. By Lemma \ref{firstlemma}, there exist $r\in \mathbb{N}$, $s_1, \ldots, s_r\in (K:M)$ and $m_1, \ldots, m_r\in N$ such that  $U=\bigcup_{i=1}^r X_{s_i}$ and for each $i=1, \ldots, r$ and each $P\in X_{s_i}$, we have $\gamma_{_{(P:M)}}=m_i/s_i$. Fix $i\in \{1, \ldots, r\}$. Then for each $P\in X_{s_i}$, $$s_i\gamma_{_{(P:M)}}=\frac{s_im_i}{s_i}\in N_{_{(P:M)}}.$$ This means that $$\rho _{UX_{s_i}}(s_i\gamma)=s_i\rho _{UX_{s_i}}(\gamma)=\epsilon_N^{X_{s_i}}(m_i)=\rho _{UX_{s_i}}(\epsilon_N^U(m_i)).$$ Thus, $s_i\gamma-\epsilon_N^U(m_i)\in \ker (\rho _{UX_{s_i}})=\Gamma _{(s_iM:M)}(\mathcal{A}(N:M)(U))$ by Proposition~\ref{2019}. Hence, there exists $n_i\in \mathbb{N}$ such that $(s_iM:M)^{n_i}(s_i\gamma-\epsilon_N^U(m_i))=0$. Then $s_i^{n_i}(s_i\gamma-\epsilon_N^U(m_i))=0$. Define $n:=\max\{n_1, n_2, \ldots, n_r\}+1$. Then for all $i=1, \ldots, r$, we have $$s_i^n\gamma=s_i^{n-1-n_i}s_i^{n_i}s_i\gamma=s_i^{n-1-n_i}s_i^{n_i}\epsilon_N^U(m_i)=\epsilon_N^U(s_i^{n-1}m_i)\in \epsilon_N^U(N).$$ Note that $$X\setminus V\left(\left(\sum_{i=1}^r Rs_i^n\right)M\right)=X\setminus V\left(\left(\sum_{i=1}^r Rs_i\right)M\right)=\bigcup_{i=1}^r X_{s_i}=U=X\setminus V(K).$$ This implies that $V((\sum_{i=1}^r Rs_i^n)M)=V(K)$. Since $M$ is faithful  primeful, it follows that $V(\sum_{i=1}^r Rs_i^n)\subseteq V((K:M))$. Hence, $$(K:M)\subseteq \sqrt{(K:M)}\subseteq \sqrt{\sum_{i=1}^r Rs_i^n}.$$ Thus, there exists $h\in \mathbb{N}$ such that $(K:M)^h\subseteq \sum_{i=1}^r Rs_i^n$, and $$(K:M)^h\gamma \subseteq \left(\sum_{i=1}^r Rs_i^n\right)\gamma=\sum_{i=1}^r Rs_i^n\gamma\subseteq \epsilon_N^U(N).$$
\end{proof}

\begin{rem}\label{natpattat}
Let $h: L\rightarrow N$ be a homomorphism of $R$-modules and let $U$ be an open subset of $X$. Now $h$ induces an $R_\p$-homomorphism $h_\p: L_\p\rightarrow N_\p$, for each $\p\in \Supp(U)$. So, we have the induced $\coprod_{\p\in \Supp(U)} R_\p$-homomorphism $$\coprod_{\p\in \Supp(U)} h_\p: \coprod_{\p\in \Supp(U)} L_\p\longrightarrow \coprod_{\p\in \Supp(U)} N_\p.$$ Thus, $h$ induces the $\mathcal{A}(R, M)(U)$-homomorphism $$\mathcal{A}(h, M)(U): \mathcal{A}(L, M)(U)\longrightarrow \mathcal{A}(N, M)(U)$$ for which $$\mathcal{A}(h, M)(U)((\gamma_\p)_{\p\in \Supp(U)})=(h_\p\gamma_\p)_{\p\in \Supp(U)}$$ for all $(\gamma_\p)_{\p\in \Supp(U)}\in \mathcal{A}(L, M)(U)$.
 It is clear that $\mathcal{A}(id_L, M)(U)=id_{\mathcal{A}(L, M)(U)}$, and if $g: N\rightarrow B$ is another homomorphism of $R$-modules, then $\mathcal{A}(g\circ h, M)(U)=\mathcal{A}(g, M)(U)\circ \mathcal{A}(h, M)(U)$. Thus, $\mathcal{A}(\bullet, M)(U)$ is a covariant functor from the category of $R$-modules and $R$-homomorphisms to the category of $\mathcal{A}(R, M)(U)$-modules and $\mathcal{A}(R, M)(U)$-homomorphisms. Note that an $\mathcal{A}(R, M)(U)$-module can be regarded as an $R$-module by means of $\epsilon_R^U$. Moreover, the diagram
$$ \xymatrix{
&L\ar[rr]^h\ar[d]^{\epsilon_L^U}&&N\ar[d]^{\epsilon_N^U}&\\
& \mathcal{A}(L, M)(U)\ar[rr]^{\mathcal{A}(h, M)(U)}& & \mathcal{A}(N,M)(U)&\\}$$

commutes, and so $\epsilon^U : Id \longrightarrow \mathcal{A}(\bullet, M)(U)$ is a natural transformation of functors from the category of $R$-modules and $R$-homomorphisms to itself.
\end{rem}

The next theorem is one of the main results of this paper and shows that there is an isomorphism between the sections of the sheaf $\mathcal{A}(N, M)$ over the open set $U$ and  the ideal transform module.

\begin{thm}\label{mainloc}
Let $R$ be a Noetherian ring and let $N$ be an $R$-module. Set $U=X\setminus V(K)$, where $K\leq M$. There is a unique $R$-isomorphism $$f_{K,N}: \mathcal{A}(N,M)(U) \rightarrow  D_{(K:M)}(N):=\dlim \Hom( (K:M)^n, N)$$ such that the diagram
      $$ \xymatrix{
& N\ar[rr]^{\epsilon_N^U}\ar[drr]& & \mathcal{A}(N,M)(U)\ar[d]^{f_{K,N}}&\\
&&& D_{(K:M)}(N)&\\}$$ commutes. Moreover, if $h:L \rightarrow N$ is a homomorphism of $R$-modules, then the diagram
$$ \xymatrix{
&\mathcal{A}(L,M)(U)\ar[rr]^{\mathcal{A}(h,M)(U)}\ar[d]\ar[d]^{f_{K,L}} & & \mathcal{A}(N,M)(U)\ar[d]^{f_{K,N}}&\\
&D_{(K:M)}(L)\ar[rr]^{D_{(K:M)}(h)}&& D_{(K:M)}(N)&\\}$$ commutes, and so $f_{K}: \mathcal{A}(\bullet,M)(U) \rightarrow  D_{(K:M)}$ is a natural equivalence of functors from the category of $R$-modules and $R$-homomorphisms to itself.
\end{thm}
\begin{proof}
By Propositions \ref{20.1.8iii} and  \ref{20.1.10}, both the kernel and cokernel of $\epsilon_N^U$ are $(K:M)$-torsion. Therefore, there is a unique $R$-homomorphism $f_{K,N}: \mathcal{A}(N,M)(U) \rightarrow  D_{(K:M)}(N)$ such that the diagram $$ \xymatrix{& N\ar[rr]^{\epsilon_N^U}\ar[drr]& & \mathcal{A}(N,M)(U)\ar[d]^{f_{K,N}}&\\ &&& D_{(K:M)}(N)&\\}$$ commutes by \cite[Corollary 2.2.13(ii)]{sb98}.

By Lemma~\ref{20.1.8ii}, $\Gamma_{(K:M)}(\mathcal{A}(N,M)(U))=0$. So,  $f_{K,N}$ is injective by \cite[Corollary~2.2.13(ii)]{sb98}.

To show that $f_{K,N}$ is surjective, let $y\in D_{(K:M)}(N)$. Then there exist  $n\in \mathbb{N}$ and $h\in \Hom_R( (K:M)^n, N)$ such that $y$ is the natural image of $h$ in $D_{(K:M)}(N)$. For each $\p \in \Supp(U)$, there exists $s_{\p}\in (K:M)\setminus \p$. Hence, $$\delta:=\left( \frac{h(s_\p^n)}{s_\p^n}\right)_{\p\in \Supp(U)}\in \coprod_{\p\in \Supp(U)} N_\p.$$ Let $Q\in U$. Then for each $P\in U\setminus V(s_{(Q:M)}M)$ we have $$\frac{h(s_{(P:M)}^n)}{s_{(P:M)}^n}=\frac{s_{(Q:M)}^nh(s_{(P:M)}^n)}{s_{(Q:M)}^ns_{(P:M)}^n}=\frac{h(s_{(Q:M)}^ns_{(P:M)}^n)}{s_{(Q:M)}^ns_{(P:M)}^n}=\frac{s_{(P:M)}^nh(s_{(Q:M)}^n)}{s_{(P:M)}^ns_{(Q:M)}^n}=\frac{h(s_{(Q:M)}^n)}{s_{(Q:M)}^n}\in N_\p.$$ Thus, $\delta\in \mathcal{A}(N,M)(U)$. We claim that $f_{K, M}(\delta)=y$.

For each $r\in (K:M)^n$, we have $$r\delta:=\left( r\frac{h(s_\p^n)}{s_\p^n}\right)_{\p\in \Supp(U)}=\left( \frac{h(rs_\p^n)}{s_\p^n}\right)_{\p\in \Supp(U)}=\left( \frac{s_\p^nh(r)}{s_\p^n}\right)_{\p\in \Supp(U)}=\epsilon_N^Uh(r).$$ Hence, $rf_{K, M}(\delta)=f_{K, M}(r\delta)=f_{K, M}\epsilon_N^Uh(r)=\eta_{(K:M), N}(h(r))$, and this is just the natural image in $D_{(K:M)}(N)$ of the homomorphism $h'\in \Hom_R((K:M)^n, N)$ given by $h'(r')=r'h(r)=rh(r')$ for all $r'\in (K:M)^n$. Thus, $rf_{K, M}(\delta)=ry$ for all $r\in (K:M)^n$. Hence, $f_{K, M}(\delta)-y\in \Gamma_{(K:M)}(D_{(K:M)}(N))$, which is zero by \cite[Corollary 2.2.8(iv)]{sb98}. Therefore, $f_{K, M}$ is surjective, and so is isomorphism.

To prove the second part, it is enough to show that  in the diagram

$$\xymatrix{
&L\ar[rrr]^h\ar[dr]^{\epsilon_L^U}\ar[dddr]_{\eta_{(K:M), L}}& & & N\ar[dr]^{\epsilon_N^U}\ar[dddr]_{\eta_{(K:M), N}}\\
&& \mathcal{A}(L, M)(U)\ar[rrr]^{\mathcal{A}(h, M)(U)}\ar[dd]^{f_{K,L}}& && \mathcal{A}(N,M)(U)\ar[dd]^{f_{K,N}}&\\
&&&&&\\
&&D_{(K:M)}(L)\ar[rrr]^{D_{(K:M)}(h)}&&& D_{(K:M)}(N)&\\}$$
the front square commutes. In view of the first part of this proof we infer that
the two side triangles commute. Furthermore, the top square commutes by Remark~\ref{natpattat}, while the sloping rectangle on the underside commutes because $\eta_{(K:M)}$ is a natural transformation. Therefore,

\begin{eqnarray*}
  f_{K, N}\circ \mathcal{A}(h, M)(U)\circ \epsilon_L^U &=& f_{K, N}\circ \epsilon_N^U\circ h\\
   &=&\eta_{(K:M), N}\circ h\\
   &=&D_{(K:M)}(h)\circ \eta_{(K:M), L}\\
   &=& D_{(K:M)}(h)\circ f_{K, L}\circ \epsilon_L^U.
\end{eqnarray*}
However, by \cite[Corollary~2.2.11(ii)]{sb98}, Propositions~\ref{20.1.8iii} and \ref{20.1.10}, there is a unique $R$-homomorphism $h' : \mathcal{A}(L, M)(U) \longrightarrow D_{(K:M)}(N)$ such that the diagram
$$ \xymatrix{
&  L \ar[rr]^{\epsilon_L^U}\ar[d]^h & & \mathcal{A}(L,M)(U)\ar[d]^{h'}&\\
&N\ar[rr]^{\eta_{(K:M), N}} & & D_{(K:M)}(N)& \\}$$
commutes, and so $f_{K, N}\circ \mathcal{A}(h, M)(U)=D_{(K:M)}(h)\circ f_{K, L}$, as required.
\end{proof}

Theorem~\ref{mainloc} has useful consequences which help us to find more information about the sheaf $\mathcal{A}(N, M)$. We list some of these consequences in the following corollaries.

\begin{cor}\label{cor1d}
Let  $R$ be a Noetherian ring and let $N$ be an $R$-module. Set $U=X\setminus V(K)$, where $K\leq M$. Then the following hold:
\begin{enumerate}
  \item $\mathcal{A}(\Gamma_{(K:M)}(N), M)(U)=0$;
  \item $\mathcal{A}(N, M)(U)\cong \mathcal{A}(N/\Gamma_{(K:M)}(N), M)(U)$;
  \item $\mathcal{A}(N, M)(U)\cong \mathcal{A}(\mathcal{A}(N, M)(U), M)(U)$;
  \item $H^1_{(K:M)}(\mathcal{A}(N, M)(U))=0$;
  \item $H^i_{(K:M)}(N)\cong H^i_{(K:M)}(\mathcal{A}(N, M)(U)) $ for all $i>1$.
\end{enumerate}
\end{cor}
\begin{proof}
Use Theorem~\ref{mainloc} and \cite[Corollary~2.2.8]{sb98}.
\end{proof}

\begin{cor}\label{cor2d}
Suppose that $R$ is a Noetherian ring. Set $U=X\setminus V(K)$, where $K\leq M$. Let $N$ be a $(K:M)$-torsion $R$-module. Then $\mathcal{A}(N, M)(U)=0$.
\end{cor}
\begin{proof}
Use Corollary~\ref{cor1d}(1).
\end{proof}

\begin{eg}\label{gerdogholombeh}
Consider the $\mathbb{Z}$-modules $N=\mathbb{Z}/8\mathbb{Z}$ and $M= \mathbb{Q} \oplus (\oplus_{p\in \Omega} \frac{\mathbb{Z}}{p\mathbb{Z}})$, where $\Omega$ is the set of all prime integers. Let $q$ be a prime integer. Then $qM=\mathbb{Q} \oplus (\oplus_{p\in \Omega\setminus \{q\}} \frac{\mathbb{Z}}{p\mathbb{Z}})\neq M$. Hence, by \cite[Proposition 2]{lu84}, $qM$ is a $(q)$-prime submodule of $M$. Indeed, $qM$ is a maximal submodule of $M$ and moreover, $T(M)=(0)\oplus (\oplus_{p\in \Omega} \frac{\mathbb{Z}}{p\mathbb{Z}})$. Hence, by \cite[Lemma 4.5]{lu03}, $T(M)$ is a $(0)$-prime submodule of $M$. Therefore, $\Spec(M)=\{qM \,| \, q \text{ is a prime integer}\}\cup \{(0)\oplus (\oplus_{p\in \Omega} \frac{\mathbb{Z}}{p\mathbb{Z}})\}$. This shows that $M$ is primeful. Now let $K=\mathbb{Q} \oplus (\oplus_{p\in \Omega\setminus \{2\}} \frac{\mathbb{Z}}{p\mathbb{Z}})$ and $U=\Spec(M)\setminus V(K)$. Then by Corollary~\ref{cor2d}, $\mathcal{A}(\mathbb{Z}/8\mathbb{Z}, \mathbb{Q} \oplus (\oplus_{p\in \Omega} \frac{\mathbb{Z}}{p\mathbb{Z}}))(U)=0$.
\end{eg}

In the next corollary, we are going to obtain a particularly simple description of the section of the sheaf $\mathcal{A}(N, M)$ over the open set $U$, in terms of objects which are perhaps more familiar, in the case in which $R$ is $PID$.

\begin{cor}
Suppose that $R$ is a $PID$ and $N$  is an $R$-module. Set $U=X\setminus V(K)$, where $K\leq M$.    Then there exists $a\in R$ such that $\mathcal{A}(N, M)(U)=N_a$.
\end{cor}
\begin{proof}
Since $R$ is a $PID$, there exists an element $a\in R$ such that $Ra=(K:M)$. By Theorem~\ref{mainloc} and \cite[Theorem 2.2.16]{sb98}, we have  $\mathcal{A}(N, M)(U)\cong D_{(K:M)}(N)\cong N_a$.
\end{proof}

Let $H$ be an $R$-module. An element $a\in R$ is said to be $H$-regular if $ax \not= 0$ for all $0\neq x\in H$. A sequence $a_1, \ldots, a_n$ of elements of $R$ is an $H$-sequence (or an $H$-regular sequence) if the following two conditions hold: (1) $a_1$ is $H$-regular, $a_2$ is $(H/a_1H)$-regular, $\ldots$, $a_n$ is $(H/\sum_{i=1}^{n-1}a_i H)$-regular; (2) $H/\sum_{i=1}^{n}a_i H\neq 0$.

\begin{cor}
Let $R$ be a Noetherian ring and   $N$ be an $R$-module. Set $U=X\setminus V(K)$, where $K\leq M$. Assume that $(K:M)$ contains an $N$-sequence of length 2. Then $\mathcal{A}(N,M)(U)\cong N $.
\end{cor}
\begin{proof}
Use Theorem \ref{mainloc} and \cite[Corollary 2.2.6]{sb98}.
\end{proof}

\begin{cor}\label{corloc}
Let $R$ be a Noetherian ring. Set $U=X\setminus V(K)$, where $K\leq M$. Then for each $n\in \mathbb{N}$, the functors  $\mathcal{R}^n\mathcal{A}(\bullet,M)(U)$ and $H^{n+1}_{(K:M)}(\bullet)$ are naturally equivalent, where the functors $\mathcal{R}^n\mathcal{A}(\bullet,M)(U)$ are the right derived functors of $\mathcal{A}(\bullet,M)(U)$. In particular, for each $R$-module $N$, $$\mathcal{R}^n\mathcal{A}(N, M)(U) \cong H^{n+1}_{(K:M)}(N).$$
\end{cor}
\begin{proof}
Use Theorem \ref{mainloc} and \cite[Theorem 2.2.4]{sb98}.
\end{proof}

Finally, we introduce  a scheme over the prime spectrum of  the $R$-module $M$.   We recall that a scheme $Y$ is \emph{locally Noetherian} if it can be covered by open affine subsets $\Spec(A_i)$, where each $A_i$ is a Noetherian ring. Also, $Y$ is \emph{Noetherian} if it is locally Noetherian and quasi-compact (see \cite{hart}).

\begin{prop}\label{prop.mainprop}
Let $N$ be an $R$-module. For any element $f\in R$, the module $\mathcal{A}(N,M)(X_f)$ is isomorphic to the localized module $N_f$.
\end{prop}
\begin{proof}
We define the map $\Theta : N_f\rightarrow \mathcal{A}(N,M)(X_f)$ by $\frac{a}{f^m}\mapsto (\frac{a}{f^m})_{\p\in \Supp(X_f)}$. Indeed,  $\Theta$ takes $\frac{a}{f^m}$ to the $\frac{a}{f^m}\in N_\p$, for each $\p\in \Supp(X_f)$. It is easy to see that $\Theta$ is a well-defined homomorphism. We claim that $\Theta$ is an isomorphism.

If $\Theta(\frac{a}{f^n})=\Theta(\frac{b}{f^m})$, then for every $P\in X_f$, $\frac{a}{f^n}$ and $\frac{b}{f^m}$ have the same image in $N_\p$, where $\p=(P:M)$. Thus, there exists $h\in R\setminus \p$ such that $h(f^m a-f^n b)=0$ in $N$. Let $I=(0:_R f^m a-f^n b)$. Then $h\in I$ and $h\not\in \p$, so $I\nsubseteq \p$. This happens for any $P\in X_f$, so we conclude that $$V(I)\cap \Supp(X_f)=\emptyset.$$ Hence, $$\Supp(X_f)\subseteq D(I):=\Spec(R)\setminus V(I).$$ Since $M$ is faithful primeful, $f\in \sqrt{I}$ and so, $f^l \in I$ for some positive integer $l$. Now, we have $f^l(f^m a-f^n b)=0$ which shows that $\frac{a}{f^n}=\frac{b}{f^m}$ in $N_\p$. Therefore,  $\Theta$ is injective.

Let $(\gamma_\p)_{\p\in\Supp(X_f)}\in \mathcal{A}(N,M)(X_f)$. Then we can cover $X_f$ with open subsets $V_i$, on which $\gamma_{(P:M)}$ is represented by $\frac{a_i}{g_i}$,  with $g_i\not\in (P:M)$ for all $P\in V_i$, in other words $V_i \subseteq X_{g_i}$. By \cite[Proposition 4.3]{lu99}, the open sets in the form $X_h$ form a base for the topology. So, we may assume that $V_i=X_{h_i}$ for some $h_i\in R$. Since $X_{h_i}\subseteq X_{g_i}$, we have $h_i\in \sqrt{Rg_i}$. Thus $h_i^n\in Rg_i$ for some $n\in \mathbb{N}$. So, $h_i^n=cg_i$ and $$\frac{a_i}{g_i}=\frac{ca_i}{cg_i}=\frac{ca_i}{h_i^n}.$$
We see that $\gamma_{(P:M)}$ is represented by $\frac{b_i}{k_i}$, ($b_i=ca_i, k_i=h_i^n$) on $X_{k_i}$  and
$X_f$ is covered by the open subsets $X_{k_i}$. The open cover $X_f=\bigcup X_{k_i}$ has a finite subcover. Suppose that, $X_f\subseteq X_{k_1}\cup\cdots\cup X_{k_n}$. For $1\leq i, j\leq n$,  $\frac{b_i}{k_i}$ and $\frac{b_j}{k_j}$ both represent $\gamma_{(P:M)}$ on $X_{k_i}\cap X_{k_j}$. By \cite[Corollary 4.2]{lu99}, $X_{k_i}\cap X_{k_j}=X_{k_ik_j}$ and by injectivity of $\Theta$, we get $\frac{b_i}{k_i}=\frac{b_j}{k_j}$ in $N_{k_ik_j}$. Hence, for some $n_{ij}$, $$(k_i k_j)^{n_{ij}}(k_jb_i - k_ib_j)= 0.$$ Let $m=\max\{n_{ij}| 1\leq i, j\leq n\}$. Then $$k_j^{m+1}(k_i^mb_i)-k_i^{m+1}(k_j^mb_j)=0.$$  Replacing each $k_i$ by $k_i^{m+1}$, and $b_i$ by $k_i^mb_i$, we still see that $\gamma_{(P:M)}$ is represented on $X_{k_i}$ by $\frac{b_i}{k_i}$, and furthermore, we have $k_jb_i=k_ib_j$ for all $i, j$.  Moreover, $$D_f=\psi(X_f)\subseteq\bigcup_{i=1}^n \psi(X_{k_i})=\bigcup_{i=1}^n D_{k_i},$$ where $\psi$ is the natural map $\psi:\Spec(M)\rightarrow \Spec(R)$.  So, there are $c_1, \cdots, c_n$ in $R$ and $t\in \mathbb{N}$ such that $f^t=\sum_i c_ik_i$. Let $a=\sum_i c_i b_i$. Then for each $j$ we have $$k_ja=\sum_i c_i b_i k_j=\sum_i c_i k_i b_j=b_jf^t.$$ This implies that $\frac{a}{f^t}=\frac{b_j}{k_j}$ on $X_{k_j}$. Therefore, $\Theta(\frac{a}{f^t})=\gamma_{(P:M)}$ everywhere, which shows that $\Theta$ is surjective. This completes the proof.
\end{proof}

Let $N$ be an $R$-module. Then Proposition~\ref{prop.mainprop} shows that  the global sections of $\mathcal{A}(N,M)$ recover $N$.

\begin{thm}\label{scheme}
Let $X$ be a $T_0$-space. Then $(X, \mathcal{A}_{X})$ is a scheme. Moreover, if $R$ is Noetherian, then $(X, \mathcal{A}_{X})$ is a Noetherian scheme.
\end{thm}
\begin{proof}
By Corollary~\ref{cor.stalk1}, $(X, \mathcal{A}_{X})$ is a locally ringed space. Let $g\in R$. Since the natural map $\psi:\Spec(M)\rightarrow \Spec(R)$ is continuous by \cite[Proposition 3.1]{lu99}, the map $\psi|_{X_{g}}:X_{g}\rightarrow \psi(X_{g})$ is also continuous. By assumption and Remark \ref{t00}, $\psi|_{X_{g}}$ is a bijection. Let $E$ be a closed subset of $X_g$. Then  $E=X_g\cap V(N)$ for some submodule $N$ of $M$. Hence, $\psi(E)=\psi(X_g\cap V(N))=\psi(X_g)\cap V(N:M)$ is a closed subset of $\psi(X_g)$. Therefore, $\psi|_{X_g}$ is a homeomorphism.

Suppose that $X=\bigcup_{i\in I}X_{g_i}$. Since $M$ is faithful primeful and $X$ is a $T_0$-space, for each $i\in I$, $$X_{g_i}\cong \psi(X_{g_i})=\Supp(X_{g_i})=D_{g_i}\cong \Spec(R_{g_i}).$$ Thus, by Proposition~\ref{prop.mainprop}, $X_{g_i}$ is an affine scheme and this implies that $(X, \mathcal{A}_{X})$ is a scheme. For the last statement,  since  $R_{g_i}$ is Noetherian for each $i\in I$,  $(X, \mathcal{A}_{X})$ is a locally Noetherian scheme. Moreover, $X$ is a Noetherian topological space (and so is quasi-compact) by \cite[Corollary 3.25]{abbhass2012}. Therefore, $(X, \mathcal{A}_{X})$ is a Noetherian scheme.
\end{proof}

\section*{ACKNOWLEDGMENT}
The authors specially thank the referee for the helpful suggestions and comments.

\providecommand{\bysame}{\leavevmode\hbox to3em{\hrulefill}\thinspace}
\providecommand{\MR}{\relax\ifhmode\unskip\space\fi MR }
\providecommand{\MRhref}[2]{%
  \href{http://www.ams.org/mathscinet-getitem?mr=#1}{#2}
}
\providecommand{\href}[2]{#2}

\end{document}